\documentclass{article}

\usepackage{arxiv}

\usepackage[utf8]{inputenc} 
\usepackage[T1]{fontenc}    
\usepackage{hyperref}       
\usepackage{url}            
\usepackage{booktabs}       
\usepackage{amsfonts}       
\usepackage{nicefrac}       
\usepackage{microtype}      
\usepackage{lipsum}         
\usepackage{amsmath,amsfonts,amssymb,amsthm,booktabs,color,epsfig,graphicx,mathtools,dsfont,float}
\usepackage{cleveref}
\usepackage{subcaption}
\usepackage{natbib}
\usepackage{doi}
\theoremstyle{plain}
\newtheorem{theorem}{Theorem}

\newtheorem{lemma}{Lemma}

\newtheorem{assump}{ }

\theoremstyle{definition}

\newtheorem{remark}{Remark}

\title{Functional relative error regression under left truncation and right censoring}
\date{March 8, 2023}

\newif\ifuniqueAffiliation
\uniqueAffiliationtrue

\ifuniqueAffiliation 
\author{\hspace{1mm}A.~Boucetta\\
	Faculty of Mathematics\\
	USTHB\\
	Algiers, Algeria \\
	\texttt{aboucetta@usthb.dz} \\
	\And
	 Z.~Guessoum \\
	 Faculty of Mathematics \\
	 USTHB\\
	 Algiers, Algeria \\
	\texttt{zguessoum@usthb.dz} \\
	\AND
	\hspace{1mm}E.~Ould-Said \\
	Laboratoire de Math. Pures et Appliquées\\
	Univ. du Littoral Côte d'Opale\\
	Calais, France \\
	\texttt{elias.ould-said@univ-littoral.fr} \\
}
\else
\usepackage{authblk}

\newbox{\orcid}\sbox{\orcid}{\includegraphics[scale=0.06]{orcid.pdf}} 
\author[1]{%
	\href{https://orcid.org/0000-0000-0000-0000}{\usebox{\orcid}\hspace{1mm}David S.~Hippocampus\thanks{\texttt{hippo@cs.cranberry-lemon.edu}}}%
}
\author[1,2]{%
	\href{https://orcid.org/0000-0000-0000-0000}{\usebox{\orcid}\hspace{1mm}Elias D.~Striatum\thanks{\texttt{stariate@ee.mount-sheikh.edu}}}%
}
\affil[1]{Department of Computer Science, Cranberry-Lemon University, Pittsburgh, PA 15213}
\affil[2]{Department of Electrical Engineering, Mount-Sheikh University, Santa Narimana, Levand}
\fi

\renewcommand{\shorttitle}{Functional relative error regression under left truncation and right censoring}

\begin{document}
\maketitle

\begin{abstract}
The nonparametric estimators built by minimizing the mean squared relative error are gaining in popularity for their robustness in the presence of outliers in comparison to the Nadaraya Watson estimators.
In this paper we build a relative error regression function estimator in the case of a functional explanatory variable and a left truncated and right censored scalar variable. 
The pointwise and uniform convergence of the estimator is proved and its performance is assessed by a numerical study in particularly the robustness which is highlighted using the influence function as a measure of robustness.
\end{abstract}


\section{Introduction}\label{sec1}
Functional data is a generalization of the multivariate data concept (vectors) to infinite dimension. This type of data is obtained when an observed phenomenon varies over a continuum set which is often the case in fields such as physics, chemistry (spectrometric analysis), health (ECG data for example), \ldots \\
To exploit the richness of information brought by the infinite dimension structure of the data, a set of statistical tools is needed. Functional data analysis (FDA) is a statistical field that deals with the treatment and analysis of this particular type of data, it has been popularized mostly by the monographs of \cite{ramsay2005functional} and \cite{ferraty2006nonparametric}. The work of \cite{ramsay2005functional} is a good introduction to the practical aspects of FDA and to parametric functional models while the work of \cite{ferraty2006nonparametric} focuses on the nonparametric methods applied to functional data and is considered as a major reference on the topic of nonparametric functional regression which is the approach on which this work is based. For a recent account of the theory of FDA we refer the interested reader to \cite{hsing_eubank_2015}.\\
In a scalar on functional regression framework the task is to model the link between a  real random variable denoted by  $Y$ and a functional explanatory variable $\boldsymbol{\chi}$ taking its  values in  $(\mathcal{F},d)$, where $\mathcal{F}$ is a semi metric space and  $d$ its corresponding semi metric, such that $Y=r(\boldsymbol{\chi})+\epsilon ,$
where $r(\cdot)$ is the regression operator mapping $\mathcal{F}$ to $\mathbb{R}$ and $\epsilon$ an error random variable. It is worth noting that in some practical cases, the variable of interest $Y$, can not be fully observed. When we deal with time-to-event data as in survival analysis, the lifetime may be left truncated and/or right censored. Right censoring occurs when for some individuals, the event of interest is observed only if it happens prior to some specified times and these times vary from an individual to another. We have left truncation when only the individuals who experience the event of interest after a certain delayed entry time are included in the study \citep{klein2003survival}.
Usually $r$ is estimated by minimizing the least squares criterion. The estimations obtained then are sensitive to outliers since it assumes that all the variables have an equal weight, which makes them inadequate in some situations with large outliers. To overcome this limitation we shall use the following mean quadratic relative error as a loss function:
\begin{small}
\begin{equation}
\label{eqn:une}
\mathbb{E}\left[\left(\frac{Y-r(\boldsymbol{\chi} )}{Y}\right)^{2}\mid \boldsymbol{\chi} \right]\text{, for } Y>0.
\end{equation}
\end{small}
\cite{park1998relative} have shown that the minimizer of (\ref{eqn:une}) is given by:
\begin{small}
\begin{equation}
\label{eqn:sol}
r(\chi)=\frac{\mathbb{E}\left[Y^{-1}\mid \boldsymbol{\chi}=\chi\right]}{\mathbb{E}\left[Y^{-2}\mid \boldsymbol{\chi} =\chi \right]}=:\frac{r_{1}(\chi)}{r_{2}(\chi)},
\end{equation}
\end{small}
where $r_{1}(\chi )$ and $r_{2}(\chi )$ are assumed finite almost surely. The first nonparametric estimator of equation (\ref{eqn:sol}) was proposed by \cite{jones2008relative} in the case when the explanatory variable is a real variable while in the functional case we refer to \cite{demongeot2016relative}. Research that treats nonparametric functional data analysis (or NPFDA) and the relative error regression estimator with incomplete data includes but not limited to \cite{altendji2018functional}, \cite{mechab2019nonparametric} and \cite{fetitah2020strong}.\\
The aim of this paper is to bring together the areas of NPFDA, robust estimation and survival analysis by establishing the pointwise and the uniform convergence with rates of the constructed estimator that takes into account both right-censoring and left-truncation of the variable of interest and the functional character of the covariable and highlighting its asymptotic properties and robustness.\\
The rest of this paper is organized as follows. In Section~\ref{sec2} we introduce the LTRC model and its components. Section~\ref{sec3} discusses the construction of the proposed estimator. In Section~\ref{sec4} and Section~\ref{sec5} the main results and their assumptions are stated while the proofs are relegated to Section~\ref{sec7}. Section~\ref{sec6} is devoted to a deep simulation study of the performance and robustness of the relative error regression estimator using the influence function.
\section{LTRC model}\label{sec2}

Let $(\boldsymbol{\chi}_{i},Y_{i})_{i=1,\cdots,N}$ be an $N$ independent identically distributed (iid) sample of the pair $(  \boldsymbol{\chi} ,Y )$ with values in ${\cal F}\times \mathbb{R}_+^*$ where the sample size $N$ is deterministic but unknown and $Y$ represents a lifetime variable with continuous distribution function $F$.\\
In an LTRC model the variable of interest is subject  to a left truncation and a right censoring. The right censoring is induced by a random variable $S$ supposed strictly positive with continuous distribution function denoted by $G$, that is, we observe $(\boldsymbol{\chi}_{i},Z_{i},\delta_{i})$ where $Z_{i}=\min \lbrace Y_{i}, S_{i} \rbrace$ and $\delta_{i}=\mathds{1}_{\lbrace Y_{i}\leq S_{i}\rbrace}$ the indicator of non censoring, meanwhile, the left truncation is due to the random variable $T$, that is, we observe $Z_{i}$ only if $Z_{i}\geq T_{i}$. We denote by $L$ the distribution function of $T$ and by $\alpha=\mathbb{P}(Z_{i}\geq T_{i})$ the probability of absence of truncation which is assumed strictly positive (if not then nothing is observed). So ultimately one observes the quadruple $(\boldsymbol{\chi}_{i},Z_{i},T_{i},\delta_{i})$ $i=1,...,n$ with $n\leq N$ if there is no confusion in regard to the index $i$, note that conditionally on the value of $n$ the observed data are still iid.\\
We assume that the random variables $Z,T,S$ are mutually independent and it follows that the distribution function of $Z$ is given by $H=1-(1-F)(1-G)$. Let $\mathbb{P}$ denote the probability measure associated to the original $N$ sample and $\mathbf{P}$ the probability measure related to the sample observed conditionally on the event $\lbrace Z_{i}\geq T_{i}\rbrace$, also we denote by $\mathbb{E}$ and $\mathbf{E}$ their respective associated expectation operators. \\
Let $C(y)=\mathbf{P}(T\leq y \leq Z)=\mathbf{L}(y)-\mathbf{H}(y)$, where $\mathbf{L}(y)$, $\mathbf{H}(y)$ and $C(y)$ can be estimated empirically by
\begin{small}
$$\mathbf{L}_{n}(y)=\frac{1}{n}\sum_{i=1}^{n}\mathds{1}_{\lbrace T_{i}\leq y\rbrace}\; \;\text{, }\; \mathbf{H}_{n}(y)=\frac{1}{n}\sum_{i=1}^{n}\mathds{1}_{\lbrace Z_{i}\leq y\rbrace}\;\text{and }\; C_{n}(y)=\frac{1}{n}\sum_{i=1}^{n}\mathds{1}_{\lbrace T_{i}\leq y \leq Z_{i}\rbrace}.$$
\end{small} 
In a LTRC model the distribution function $F$ is estimated by the product limit estimator given in \cite{tsai1987note} called the TJW estimator defined as $F_{n}(y)=1-\prod_{Z_{i}\leq y}\left(1-\frac{1}{nC_{n}(Z_{i})}\right)^{\delta_{i}}.$
Note that $C(y)$ can be rewritten as
$
C(y)=\mathbb{P}(T\leq y \leq Z\mid Z\geq T)=\alpha^{-1}L(y)\bar{H}(y)=\alpha^{-1}L(y)\bar{F}(y)\bar{G}(y),
$
which implies
\begin{equation}
\label{eqn:alpha}
\alpha=\frac{L(y)\bar{F}(y)\bar{G}(y)}{C(y)}. 
\end{equation}
\cite{he1998estimation} showed that a consistent estimator of $\alpha$ is given by
\begin{equation}
\label{eqn:alphan}
\alpha_{n}=\dfrac{L_{n}(y)\bar{F}_{n}(y)\bar{G}_{n}(y)}{C_{n}(y)}\text{ , for any $y$ such that }C_{n}(y)\neq0 ,
\end{equation}
where $L_{n}(y)$ is the \cite{lynden1971method} product limit estimator of the distribution function L defined as $L_{n}(y)=\prod_{T_{i}>y}\left(1-\frac{1}{nC_{n}(T_{i})}\right)$ and $G_{n}(y)$ a TJW-type estimator of the distribution function $G$ defined by $G_{n}(y)=1-\prod_{Z_{i}\leq y}\left(1-\dfrac{1}{nC_{n}(Z_{i})}\right)^{1-\delta_{i}}.$
\begin{remark}
\cite{he1998estimation} showed that the value of $\alpha_{n}$ does not depend on the value of $y$.
\end{remark}
\noindent For any distribution function W denote by $a_{W}=\sup\lbrace y:W(y)>0\rbrace$ and $b_{K}=\inf\lbrace y :W( y ) <1\rbrace$ its support endpoints. Following \cite{gijbels1993strong} the distribution function $F$ is identifiable, in a LTRC setting, only under some conditions on the support of $L$ and $H$ ($a_{L}\leq a_{H}$ and $b_{L}\leq b_{H}$ ). We assume that $a_{L} < a_{H}$, $b_{L}\leq b_{H}$ and that for some $b>0$,  $b<b_{H}$. Furthermore we suppose that $(T,S)$ and $(\boldsymbol{\chi}, \; Y)$ are independent.

\section{Construction of the estimator}\label{sec3}
Under a LTRC setting, one has to account for both left truncation and right censoring effect and by combining the ideas of \cite{altendji2018functional} and \cite{mechab2019nonparametric}, we propose the following mean relative error regression estimator:
\begin{small}
\begin{equation}
\label{eqn:estim rer}
\displaystyle \widehat{r}_{n}(\chi )= \displaystyle\dfrac{ \displaystyle \sum\limits_{i=1}^{n}\dfrac{ \displaystyle \delta_{i}K( \displaystyle\frac{d(\chi , \boldsymbol{\chi}_{i})}{h})Z_{i}^{-1}}{L_{n}(Z_{i})\bar{G_{n}}(Z_{i})}}{ \displaystyle \sum\limits_{i=1}^{n}\dfrac{\displaystyle  \delta_{i}K \displaystyle (\frac{d(\chi , \boldsymbol{\chi}_{i})}{h})Z_{i}^{-2}}{L_{n}(Z_{i})\bar{G_{n}}(Z_{i})}},
\end{equation}
\end{small}
where $K$ is a kernel and $h\coloneqq h_{n} $ is a sequence of positive real numbers. For technical reasons, we rewrite~\eqref{eqn:estim rer}  as
\begin{small}
$$\displaystyle \widehat{r}_{n}(\chi )=\displaystyle \dfrac{\displaystyle \hat{r}_{1,n}(\chi)}{\hat{r}_{2,n}(\chi)},$$
\end{small}
where \begin{small}
 $\displaystyle\displaystyle \widehat{r}_{\ell,n}(\chi) \coloneqq \dfrac{1}{n \mathbb{E} \Big[K  \displaystyle\Big(\frac{d(\chi , \boldsymbol{\chi}_{i})}{h}\Big)\Big]}\displaystyle \sum\limits_{i=1}^{n} \dfrac{  \displaystyle \alpha_{n} \delta_{i}K \big( \displaystyle\frac{d(\chi , \boldsymbol{\chi}_{i})}{h}\big)Z_{i}^{-\ell}}{L_{n}(Z_{i})\bar{G_{n}}(Z_{i})} 
$
\end{small} for $\ell=1,\; 2.$ 
\section{Pointwise almost sure convergence}\label{sec4}
Let $\chi $ denote an element in $\mathcal{F}$ and $\mathcal{N}_{\chi}$ its neighborhood, for $\rho>0$ we denote by $B(\chi ,\rho)$ the closed ball with radius $\rho$ and $\chi$ its center that is $B(\chi,\rho)=\lbrace \chi^{\prime}:d(\chi,\chi^{\prime})\leq \rho\rbrace$. For easy writing we shall denote by $C$, $C^{\prime}$ $c$,... any positive constant whose value can change from time to time, in  absence of confusion.\\
In what follows, we suppose that the next assumptions hold true:
\begin{assump}
\label{Hyp:A1}
$\forall h>0$,  $\phi_{\chi}(h)=\mathbb{P}\big( \boldsymbol{\chi}  \in B(\chi,h) \big) >0$ and $\lim\limits_{h\rightarrow0} \phi_{\chi}(h)=0$.
\end{assump}
\begin{assump}
\label{Hyp:A2}
$\forall (\chi_{1}, \chi_{2})\in \mathcal{N}_{\chi}\times \mathcal{N}_{\chi} $ and  for some $k_{\ell}>0$ we have $\mid r_{\ell}(\chi_1)- r_{\ell}(\chi_{2})\mid\leq Cd^{k_{\ell}}(\chi_1,\chi_2)\text{, }\ \ \ell=1,2.  $
\end{assump}
\begin{assump}
\label{Hyp:A3}
The kernel K is a measurable function with support $\left]0,1\right[$ and satisfies:
$$0<C\mathds{1}_{\left[0,1\right]}(.)\leq K(.)\leq C^{\prime}\mathds{1}_{\left[0,1\right]}(.) <\infty .$$
\end{assump}
\begin{assump}
\label{Hyp:A4}
The bandwidth $h$ is such that $\lim_{n\longrightarrow+\infty}h=0\text{ and } \lim_{n\longrightarrow+\infty}\dfrac{\log(n)}{n\phi_{\chi}(h)}=0.$
\end{assump}
\begin{assump}
\label{Hyp:A5}
For any $m\geq2$,  the conditional moments  of  $Y^{-m}$ satisfy $\mathbb{E}\left[ Y^{-m}\mid \boldsymbol{\chi}=\chi \right] < C < \infty.$
\end{assump}
\begin{remark}
Assumption~\ref{Hyp:A1} is a classic hypothesis in functional data analysis it means that the probability to have  $\boldsymbol{\chi}$, around the element at which the regression operator is evaluated, is non null. Assumption~\ref{Hyp:A2} is used to study the bias term of our estimator while the assumptions~\ref{Hyp:A3} -~\ref{Hyp:A5} are technical conditions needed to prove our result.
\end{remark}

\noindent The following theorem gives the pointwise almost sure convergence with rate of the proposed estimator $\hat{r}_{n}(\chi)$ towards $r(\chi)$.
\begin{theorem} \label{Theo1}
Under assumptions~\ref{Hyp:A1} -~\ref{Hyp:A5} , we have:
\begin{small}
$$\mid\hat{r}_{n}(\chi)-r(\chi)\mid=\mathcal{O}\left(h^{k_{1}}\right)+\mathcal{O}\left(h^{k_{2}}\right)+\mathcal{O}\left(\sqrt{\frac{\log(n)}{n\phi_{\chi}(h)}}\right) \text{ a.s}.$$
\end{small}
\end{theorem}
\begin{remark}
Our stated rate of convergence is the same as the one obtained by \cite{demongeot2016relative} in the complete data case and by \citep{mechab2019nonparametric} and \citep{altendji2018functional} respectively for the case where the data is subject to censoring or truncation.
\end{remark}
\section{Uniform almost sure convergence}\label{sec5}
In this section we establish the uniform almost sure convergence of our estimator over some subset $\mathcal{S}_{\mathcal{F}}$ of $\mathcal{F}$. In odrer to do so, it is necessary to make some additional topological assumptions  since in the functional case, the uniform consistency is not a straightforward generalization of the pointwise convergence (see \cite{ferraty2010rate}).
\begin{assump}
\label{Hyp:U1}
There exists a  function $\phi(.)$, such that: $\forall \chi \in \mathcal{S}_{\mathcal{F}}$, $\forall h>0$, $0<C \phi(h)<\mathbb{P}( {\chi}  \in B(\chi, h))<C^{\prime}\phi(h)<\infty$.
\end{assump}
\begin{assump}
\label{Hyp:U2}
There exists $\beta>0$ such that $\forall \;  \chi_{1},\chi_{2}\in \mathcal{S}_{\mathcal{F}}^{\beta}, \quad \mid r_{\ell}(\chi_{1})- r_{\ell}(\chi_{2})\mid\leq Cd^{k_{\ell}}(\chi_1,\chi_2),$ where \\
$\mathcal{S}^{\beta}_{\mathcal{F}}=\left\lbrace  \chi \in\mathcal{F}, \exists \chi^{\prime} \in\mathcal{S}_{\mathcal{F}} \text{ such that } d\left(\chi,  \chi^{\prime}\right)\leq \beta\right\rbrace$.
\end{assump}
\begin{assump}
\label{Hyp:U3}
The kernel K is bounded and Lipschitzian on $\left[0,1\right]$.
\end{assump}
\begin{assump}
\label{Hyp:U4}
The functions $\phi(.)$ and  Kolmogorov's $\epsilon$-entropy $\psi_{\mathcal{S}_{\mathcal{F}}}(.)$ of $\mathcal{S}_{\mathcal{F}}$ are such that:
\begin{itemize}
\item[•] There exist $\eta_{0}$ such that $\forall \eta <\eta_{0}$, $\phi^{\prime}\left(\eta_{0}\right)<C$.
\item[•] For $n$ large enough, we have $\frac{\left(\log n\right)^2}{n\phi\left(h\right)}<\psi_{\mathcal{S}_{\mathcal{F}}}\left(\frac{\log n}{n}\right)<\frac{n\phi\left(h\right)}{\log n}.$
\item[•] Kolmogorov's $\epsilon$-entropy $\psi_{\mathcal{S}_{\mathcal{F}}}(.)$ of $\mathcal{S}_{\mathcal{F}}$ verifies that $\sum_{n=1}^{+\infty} \exp\left(\left(1-\gamma\right)\psi_{\mathcal{S}_{\mathcal{F}}}\left(\frac{\log n}{n}\right)\right)<\infty,$\\
for at least one $\gamma >1$. 
\end{itemize}
\end{assump}
\begin{assump}
\label{Hyp:U5}
For any $m\geq2$, $\mathbb{E}\left[ Y^{-m}\mid \boldsymbol{\chi}= \chi \right] \leq C < \infty$ for all $\chi \in\mathcal{S}_{\mathcal{F}}$ and $\inf\limits_{\chi \in \mathcal{S}_{\mathcal{F}}}r_{2}(\chi) \geq C^{\prime}>0$.
\end{assump}
\begin{remark}
Assumptions~\ref{Hyp:U1},~\ref{Hyp:U2},~\ref{Hyp:U5} are the same considered by \cite{demongeot2016relative} and they represent a reformulation of assumptions~\ref{Hyp:A1},~\ref{Hyp:A2},~\ref{Hyp:A5} in the uniform consistency context, while Assumption~\ref{Hyp:U4} controls the Kolmogorov's $\epsilon$-entropy of $\mathcal{S}_{\mathcal{F}}$ and ensures that $\dfrac{\psi_{\mathcal{S}_{\mathcal{F}}}\left(\frac{\log n}{n}\right)}{n\phi(h)}$ tends to $0$ when $n\longrightarrow\infty$ \citep{ferraty2010rate}.
\end{remark}
\begin{theorem}\label{Theo2}
Let Assumptions ~\ref{Hyp:U1} -~\ref{Hyp:U5} hold. Then we have:
\begin{small}
$$\sup_{\chi \in\mathcal{S}_{\mathcal{F}}}\mid\hat{r}_{n}(\chi )-r(\chi)\mid=\mathcal{O}\left(h^{k_{1}}\right)+\mathcal{O}\left(h^{k_{2}}\right)+\mathcal{O}\left(\sqrt{\frac{\psi_{\mathcal{S}_{\mathcal{F}}}\left(\frac{\log n}{n}\right)}{n\phi\left(h\right)}}\right) \text{ a.s}.$$
\end{small}
\end{theorem}
\begin{remark}
The rate of uniform convergence obtained is equivalent to the uniform rate in the complete data framework from \cite{demongeot2016relative} and to the one stated in the censoring case in the work of \cite{fetitah2020strong}. To our knowledge, no research paper has yet dealt with the uniform convergence in a relative error functional regression context when the variable of interest is left truncated. 
\end{remark}
\section{Simulation study}\label{sec6}
In order to investigate the performance of the estimator given by~\eqref{eqn:estim rer} a simulation study must be conducted. First of all we need to generate the necessary data.
We choose the size $n$ and we consider the next curves generated by:
\begin{equation}\label{20curves}
\boldsymbol{\chi}_{i}(t)= a\cos(2\pi t)+ b\sin(4\pi t) +c(t-0.5)(t-0.25)\text{ with } t\in \left[0,1\right],
\end{equation}
where $a$, $b$, $c$ follow the uniform law $\mathcal{U}_{\left[0,3\right]}$. All the curves  are discretized on the same grid generated from 100 equidistant measurements on $\left[0,1\right]$. We also choose the following regression operator $r$ defined by $r(\boldsymbol{\chi}_{i})=\int_{0}^{1}\boldsymbol{\chi}_{i}^{2}(t)dt + 10.$ To have $Y_{i}=r(\boldsymbol{\chi}_{i})+\epsilon_{i}$ we generate $\epsilon_{i}$ as a random noise sampled from a $\mathcal{N}(0,1)$. Next, we generate $S_{i}$ from an exponential distribution $\mathcal{E}\left(\mu\right)$ where $\mu$ is adapted to obtain the desired censoring rate. We set $Z_{i}=\min(Y_{i},S_{i})$ and $\delta_{i}=\mathds{1}_{\left\lbrace Y_{i}\leq S_{i}\right\rbrace}$ and ultimately sample $T_{i}$ from a $\mathcal{N}(\lambda,2)$, where $\lambda$  takes different values to get the desired truncation rate. Finally the observation $(\boldsymbol{\chi}_{i},Z_{i},T_{i},\delta_{i})$ is kept only if $Z_{i}\geq T_{i}$, otherwise it is rejected and the procedure is resumed from the beginning.
This algorithm is repeated until $n$ observations are reached. At the end we also obtain the value of $N$.
 
\subsection{Finite sample performance}\label{subsec1}
We have observed the behavior of our estimator $\widehat{r}_{n}(\chi )=:\hat{r}_{RER}$ under different configurations : sample size, truncation and censoring rates. We have at the same time compared it to the classical (Nadaraya-Watson) regression estimator noted $\hat{r}_{NW}(\chi)$. 
The semi-metric used in this study is the usual $L_2$ distance between two curves as a measure of proximity. For the kernel choice we picked the asymmetrical quadratic kernel defined by $
K(u)=\frac{3}{2}(1-u^{2})\mathds{1}_{\left]0,1\right[}$, and the optimal bandwidth $h$ is computed using the leave-one-out cross-validation method. 
We computed the Global Mean Squared Error (GMSE) for B replicates of the estimators and given by:
\begin{small}
\begin{equation*}
GMSE=\frac{1}{B\times m}\sum\limits_{u=1}^{B}\sum\limits_{k=1}^{m}\left(\hat{r}_{n}^{\left(u\right)}\left(\chi_{k}\right)-r\left(\chi_{k}\right)\right)^{2}.
\end{equation*}
\end{small}
Here $B=200$, $m=20$ is the number of curves at which the estimator is evaluated and $\widehat{r}_{n}^{\left(u\right)}\left({\chi}_{k}\right)$ is the estimate obtained using sample number $u$. The obtained results are reported in Table~\ref{tab1}.
\begin{table}[H]
\centering
\small{
\caption{GMSE of the relative error regression estimator and the classic regression}\label{tab1}%
\begin{tabular}{ccccc}
\toprule
Censoring Rate $\approx$ & Truncation Rate $\approx$  & Sample size & RER & Classic regression\\
\midrule
20\%    & 20\%     & 100  & 0.2936 & 3.4374 \\
20\%    & 20\%     & 300   & 0.1561 & 2.9494 \\
20\%    & 20\%     & 500  & 0.0993 & 1.1001 \\
\midrule
10\%    & 20\%     & 100  & 0.2898 & 2.4836 \\
20\%    & 20\%     & 100  & 0.2936 & 3.4374 \\
40\%    & 20\%     & 100  & 0.3577 & 4.0566 \\
\midrule
20\%    & 10\%     & 100  & 0.2930 & 3.1910 \\
20\%    & 20\%     & 100  & 0.2936 & 3.4374 \\
20\%    & 40\%     & 100  & 0.3230 & 3.8820\\
\hline
\end{tabular}
}
\end{table}

We can see from Table~\ref{tab1} that in terms of the GMSE the relative error regression estimator has a
better quality of fit than the classical regression one and a better stability since it is less affected
by high percentages of censoring and truncation which comforts the good behavior of our
estimator.
\subsection{Influence function}\label{subsec2}
In this part we shall investigate the robustness of the relative error estimator in the presence of an atypical observation  in the data using the influence function (IF).
The influence function 
measures  the effect induced by outliers on an estimate. Note that an estimator can be written as a functional of the distribution function $F$ of the data for some functional $\mathcal{T}$ (see \cite{maronna2019robust} and the examples within) that is $\mathcal{T}\left(F\right)$. For an extra data point $z_{0}\in \mathbb{R}$, define $F_{\epsilon}=\left(1-\epsilon\right)F + \epsilon \Delta_{z_{0}}$, which represents  the distribution $F$ contaminated by a small fraction $\epsilon$ of outliers with $\Delta_{z_{0}}$ being the Dirac mass at the added point $z_{0}$. Hence,  the influence function  defined by:
\begin{small}
\begin{equation*}
IF\left(z_{0},\mathcal{T}\left(F\right)\right)=\lim\limits_{\epsilon \longrightarrow 0^{+}}\dfrac{\mathcal{T}\left(F_{\epsilon}\right)-\mathcal{T}\left(F\right)}{\epsilon},
\end{equation*}
\end{small}
represents the sensitivity of $\mathcal{T}\left(F\right)$ to $z_{0}$. From now on, the functional $\mathcal{T}$ is taken to be $\widehat{r}_{RER}$ or $\widehat{r}_{NW}\left(\chi\right)$. The influence function can be seen as an asymptotic version of the standardized sensitivity curve (SSC) defined by $$SSC_{n}\left(z_{0},\hat{r}_{n}\left(\chi_{0}\right)\right)=\left(n+1\right)\left(\widehat{r}_{n+1}\left(\chi_{0};z_{1},\cdots,z_{n},z_{0}\right)-\widehat{r}_{n}\left(\chi_{0};z_{1},\cdots,z_{n}\right)\right)$$ where $\chi_{0}$ is a fixed curve in $\mathcal{F}$ and $\widehat{r}_{n+1}\left(\chi_{0};z_{1},\cdots,z_{n},z_{0}\right)$(resp. $\widehat{r}_{n}\left(\chi_{0};z_{1},\cdots,z_{n}\right)$)is the estimator built using the (n+1)-sample $z_{1},\cdots,z_{n},z_{0}$ (resp. the (n)-sample $z_{1},\cdots,z_{n},z_{0}$) evaluated at $\chi_{0}$ . Indeed, if we add the extra data point $z_{0}$ to the n-sample $z_{1},\cdots,z_{n}$ therefore the fraction of contamination $\epsilon$ is equal to $\frac{1}{n+1}$. Then define the empirical influence function as $EIF_{\widehat{r}_{n}}(\chi)=SSC_{n}(z_{0},\widehat{r}_{n}\left(\chi\right))$. Note that for a robust estimator (meaning that an outlier has a very limited effect on the estimator) the empirical influence function must be close to 0.
\subsubsection{Effect of an outlier}
In order to study the impact of an outlier $z_{0}$ on the relative error estimator and the classic regression one, we shall compare the $EIF_{\widehat{r}_{n}}(\chi)$ obtained for both estimators after generating  20 curves as in (\ref{20curves}) and   a curve $\chi_{0}$. Then we take (arbitrarily) $z_{0}=300$ with $\delta_{0}=1$ (the  non-censoring indicator), and we plot the obrained $EIF$  against the distance between each curve and the curve $\chi_{0}$ for  $n=300$, the censoring rate $CR= 20\%,40\%$ and the truncation rate  $TR=20\%, 40\%$.
\begin{figure}[H]%
\centering
\begin{subfigure}[b]{0.3\textwidth}
\includegraphics[width=\textwidth]{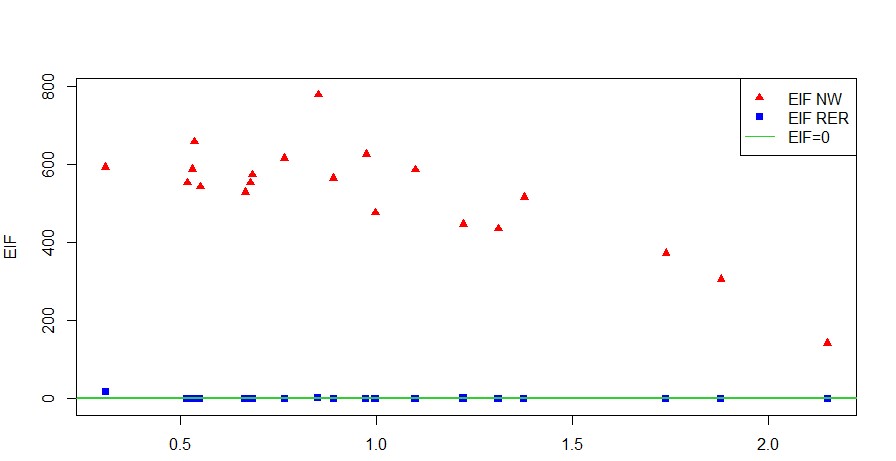}
\subcaption{}
\end{subfigure}
    \begin{subfigure}[b]{0.3\textwidth}
        \includegraphics[width=\textwidth]{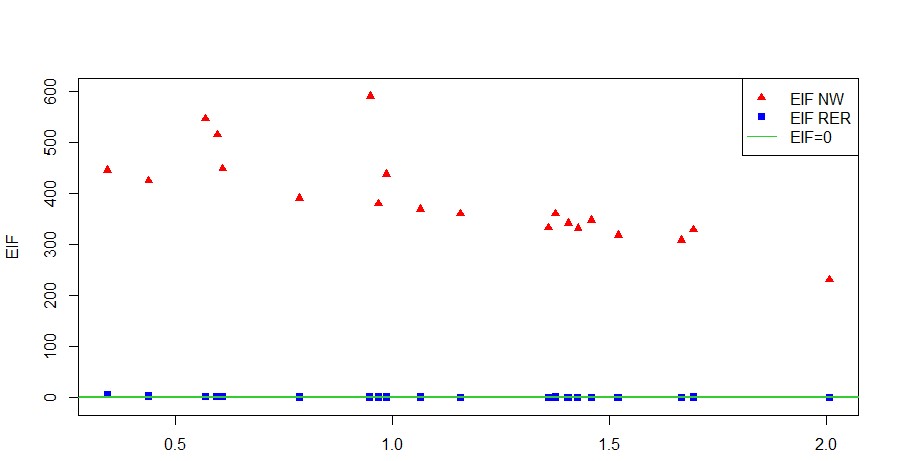}
				\subcaption{}
    \end{subfigure}
    \begin{subfigure}[b]{0.3\textwidth}
        \includegraphics[width=\textwidth]{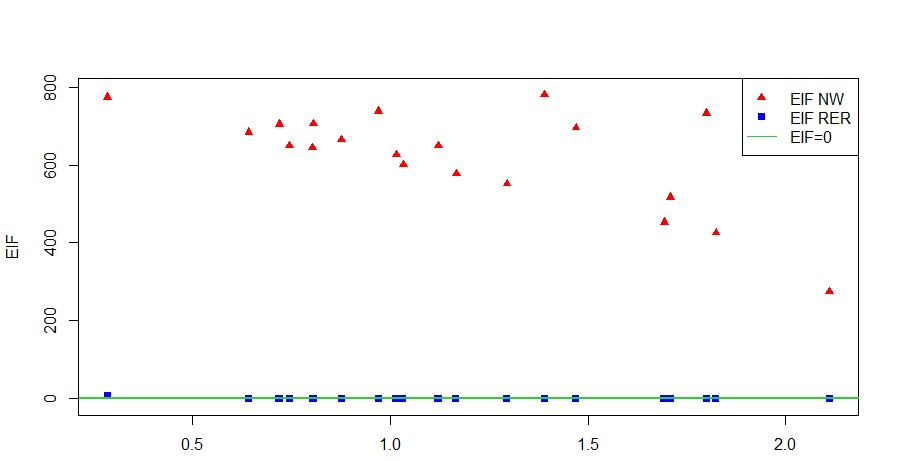}
				\subcaption{}
    \end{subfigure}
    \caption{Empirical influence function in the neighborhood of the curve $\chi_{0}$ for a single outlier $z_{0}=300$  for n=300, and (a)  $CR\approx 20\% TR \approx 20\%$, (b)  $CR\approx 20\% TR \approx 40\%$, (c)  $CR\approx 40\% TR \approx 20\%$ }\label{fig5}
\end{figure}
\section{Proofs}\label{sec7}
\subsection*{Proof of Theorem 1}
\noindent The proof of Theorem~\ref{Theo1} is based on the following decomposition and the results from Lemma~\ref{Lem1}, Lemma~\ref{Lem2} and Lemma~\ref{Lem3} below.:
\begin{small}
\begin{align}
\mid\widehat{r}_{n}(\chi)-r(\chi)\mid&\leq\dfrac{1}{\widehat{r}_{2,n}(\chi)}\Big( \mid\widehat{r}_{1,n}(\chi)-\tilde{r}_{1,n}(\chi)\mid+\mid\tilde{r}_{1,n}(\chi)-\mathbf{E}(\tilde{r}_{1,n}(\chi))\mid +\mid\mathbf{E}(\tilde{r}_{1,n}(\chi))-r_{1}(\chi)\mid\Big)\nonumber &\\
&+\dfrac{r(\chi)}{\widehat{r}_{2,n}(\chi)}\Big( \mid r_{2}(\chi)-\mathbf{E}(\tilde{r}_{2,n}(\chi))\mid +\mid\mathbf{E}(\tilde{r}_{2,n}(\chi))-\tilde{r}_{2,n}(\chi)\mid+\mid\tilde{r}_{2,n}(\chi)-\widehat{r}_{2,n}(\chi)\mid\Big). \label{eqn}&
\end{align}
\end{small}
Where
\begin{small}
\begin{align*}
\tilde{r}_{\ell,n}(\chi)=\dfrac{1}{n \mathbb{E}\left(K( \frac{d(\chi, \boldsymbol{\chi}_{1})}{h})\right)}\sum_{i=1}^{n}\frac{\alpha\delta_{i}K(\frac{d(\chi , \boldsymbol{\chi}_{i})}{h})Z_{i}^{-\ell}}{L(Z_{i})\bar{G}(Z_{i})} \text{ for } \ell=1,2.
\end{align*}
\end{small} 
\begin{lemma}
\label{Lem1}
Under  assumptions ~\ref{Hyp:A1} -~\ref{Hyp:A4} we have for $\ell=1,2$ :
$$\mid \mathbf{E}(\tilde{r}_{\ell,n}(\chi))-r_{\ell}(\chi)\mid = \mathcal{O}\left(h^{k_{\ell}}\right).$$
\end{lemma}
\begin{proof}
We have that \begin{small}$\mid \mathbf{E}(\tilde{r}_{\ell,n}(\chi))-r_{\ell}(\chi)\mid = \left\vert\dfrac{1}{\mathbb{E}(K(\frac{d(\chi,  \boldsymbol{\chi}_{1})}{h}))}\mathbf{E}\left(\dfrac{\alpha\delta_{1}K(\frac{ d(\chi,  \boldsymbol{\chi}_{1}) }{h} )Z_{1}^{-\ell}}{L(Z_{1})\bar{G}(Z_{1})}\right)-r_{\ell}(\chi)\right\vert,$\end{small} and we can show that:
\begin{small}
\begin{align}\label{star1}
\mathbf{E}\left(\dfrac{\alpha\delta_{1}K(\frac{d(\chi,  \boldsymbol{\chi}_{1})}{h})Z_{1}^{-\ell}}{L(Z_{1})\bar{G}(Z_{1})}\right)=\mathbb{E}\left(K(\frac{d(\chi, \boldsymbol{\chi}_{1})}{h})r_{\ell}(\boldsymbol{\chi}_{1})\right).
\end{align}
\end{small}
Hence and under Assumption~\ref{Hyp:A2} we obtain:
\begin{small}
\begin{align*}
\mid \mathbf{E}(\tilde{r}_{\ell,n}(\chi))-r_{\ell}(\chi)\mid\leq \dfrac{1}{\mathbb{E}(K(\frac{d(\chi, \boldsymbol{\chi}_1)}{h}))}\mathbb{E}\left[K(\frac{d(\chi, \boldsymbol{\chi}_1)}{h})Cd^{k_{\ell}}(\chi, \boldsymbol{\chi}_{1})\right]\leq Ch^{k_{\ell}} .
\end{align*}
\end{small}
\end{proof}
\begin{lemma}
\label{Lem2}
Under assumptions ~\ref{Hyp:A1},~\ref{Hyp:A3} -~\ref{Hyp:A5} we have for $\ell=1,2$:
\begin{small}
$$\mid \tilde{r}_{\ell,n}(\chi)-\mathbf{E}(\tilde{r}_{\ell,n}(\chi)) \mid = \mathcal{O}\left(\sqrt{\frac{\log(n)}{n\phi_{\chi}(h)}}\right) \text{ a.s}.$$ 
\end{small}
\end{lemma}
\begin{proof}
Set $\tilde{r}_{\ell,n}(\chi)-\mathbf{E}(\tilde{r}_{\ell,n}(\chi))=\frac{1}{n}\sum\limits_{i=1}^{n}\frac{1}{\mathbb{E}\left(K(\frac{d(\chi, \boldsymbol{\chi}_1)}{h})\right)}\left[\frac{\alpha\delta_{i}K(\frac{d(\chi, \boldsymbol{\chi}_{i})}{h})Z_{i}^{-\ell}}{L(Z_{i})\bar{G}(Z_{i})}-\mathbf{E}\left(\frac{\alpha\delta_{i}K(\frac{d(\chi, \boldsymbol{\chi}_{i})}{h})Z_{i}^{-\ell}}{L(Z_{i})\bar{G}(Z_{i})}\right)\right]=:\frac{1}{n}\sum\limits_{i=1}^{n}W_{i,\ell}$.
We use an exponential type inequality given in Corrolary A8.ii in \cite{ferraty2006nonparametric}, and for that  we need to evaluate $\mathbf{E}(\mid W_{i,\ell}^{m}\mid)$ for $m\geq2$. Assumptions \ref{Hyp:A3}, \ref{Hyp:A5} and the binomial expansion formula allow to show that:
\begin{small}
\begin{align*}
\mathbf{E}(\mid W_{i,\ell}^{m}\mid)&\leq \max\limits_{k\in \lbrace0,1,2,...,m\rbrace}\left\lbrace\binom{m}{k}C^{m-k}\right\rbrace\sum\limits_{k=0}^{m}\mathbf{E}\left\vert\frac{\left(\frac{\alpha\delta_{i}K(\frac{d(\chi, \boldsymbol{\chi}_{i})}{h})Z_{i}^{-\ell}}{L(Z_{i})\bar{G}(Z_{i})}\right)}{\mathbb{E}\left(K(\frac{d(\chi, \boldsymbol{\chi}_)}{h})\right)}\right\vert^{k}&\\
&\leq C\max\limits_{k\in \lbrace0,1,2,...,m\rbrace}\phi_{\chi}^{-k+1}(h)\leq C\phi_{\chi}^{-m+1}(h).
\end{align*}
\end{small}
We apply the mentioned inequality with $a=\phi_{\chi}^{-\frac{1}{2}}(h)$ and $u_{n}=\frac{a^{2}\log(n)}{n}$, and we obtain by setting $\epsilon=\epsilon_{0}\sqrt{u_{n}}$:
\begin{small}
\begin{align*}
\mathbf{P}\left(\mid \tilde{r}_{\ell,n}(\chi)-\mathbf{E}(\tilde{r}_{\ell,n}(\chi)) \mid>\epsilon_{0}\sqrt{u_{n}}\right)=\mathbf{P}\left(\frac{1}{n}\left\vert\sum_{i=1}^{n} W_{i,\ell}\right\vert>\epsilon_{0}\sqrt{u_{n}}\right)\leq 2 \exp\left(-\frac{\epsilon_{0}^{2}\log(n)}{2(1+\epsilon_{0}\sqrt{u_{n}})}\right)\leq 2n^{-C\epsilon_{0}}.
\end{align*}
\end{small}
An appropriate choice of $\epsilon_{0}$ and Borel-Cantelli's lemma allow us to conclude the proof.
\end{proof}
\begin{lemma}
\label{Lem3}
Under assumptions ~\ref{Hyp:A1},~\ref{Hyp:A3} -~\ref{Hyp:A5} we have for $\ell=1,2$:
\begin{small}
$$ \mid\widehat{r}_{\ell,n}(\chi)-\tilde{r}_{\ell,n}(\chi)\mid=\mathcal{O}\left((n^{-1}\log\log(n))^{1/2})\right) \text{ a.s}.$$
\end{small}
\end{lemma}
\begin{proof}
When replacing $\alpha,\alpha_{n}$ by their expressions~\eqref{eqn:alpha} and~\eqref{eqn:alphan}, we have
\begin{small}
\begin{align*}
\hspace{-1cm}\mid\widehat{r}_{\ell,n}(\chi)-\tilde{r}_{\ell,n}(\chi)\mid
&\leq \sup_{a_{H}\leq t \leq b} \Bigg\vert\frac{\bar{F}_{n}(t)}{C_{n}(t)}-\frac{\bar{F}(t)}{C(t)}\Bigg\vert\frac{1}{n\mathbb{E}\left(K(\frac{d(\chi, \boldsymbol{\chi}_)}{h})\right)}\sum_{i=1}^{n}K(\frac{d(\chi, \boldsymbol{\chi}_{i})}{h})\delta_{i}Z_{i}^{-\ell}&\\
&\leq \underbrace{\Bigg\vert\frac{\sup\limits_{a_{H}\leq t \leq b}\mid F_{n}(t)-F(t)\mid}{\inf\limits_{a_{H}\leq t \leq b}\mid C(t)\mid - \sup\limits_{a_{H}\leq t \leq b} \mid C_{n}(t)-C(t) \mid}+\frac{\sup\limits_{a_{H}\leq t \leq b} \mid C_{n}(t)-C(t) \mid}{\inf\limits_{a_{H}\leq t \leq b}\mid C(t)\mid(\inf\limits_{a_{H}\leq t \leq b}\mid C(t)\mid - \sup\limits_{a_{H}\leq t \leq b} \mid C_{n}(t)-C(t) \mid)}\Bigg\vert}_{\mathcal{I}_{1}}&\\
&\times \underbrace{\frac{1}{n\mathbb{E}\left(K(\frac{d(\chi, \boldsymbol{\chi}_)}{h})\right)}\sum_{i=1}^{n}K(\frac{d(\chi, \boldsymbol{\chi}_{i})}{h})\delta_{i}Z_{i}^{-\ell}}_{\mathcal{I}_{2}} = \mathcal{I}_{1} \times \mathcal{I}_{2}.
\end{align*}
\end{small}
On the one hand by Corollary 1 in \cite{gijbels1993strong} we have that 
\begin{small}$\sup\limits_{a_{H}\leq t \leq b}\mid F_{n}(t)-F(t)\mid=\mathcal{O}\left((n^{-1}\log\log(n))^{1/2})\right)$\end{small}. Also from Lemma 3.3 in \cite{zhou1999nonparametric} we have\begin{small}$\sup\limits_{a_{H}\leq t \leq b} \mid C_{n}(t)-C(t) \mid =\mathcal{O}\left((n^{-1}\log\log(n))^{1/2})\right)$\end{small} a.s, and using the fact that $C(t)> 0$ so $\exists$ $\beta>0$ such that $C(t)\geq\beta>0$ we obtain:
\begin{small}
\begin{equation}
\label{eqn1}
\mathcal{I}_{1}=\mathcal{O}\left((n^{-1}\log\log(n))^{1/2})\right) \text{ a.s}.
\end{equation}
\end{small}
On the other hand by stationarity we have:
\begin{small}
\begin{align*}
\mathcal{I}_{2}&=\underbrace{\frac{1}{n\mathbb{E}\left(K(\frac{d(\chi, \boldsymbol{\chi}_)}{h})\right)}\sum_{i=1}^{n}\left[ K(\frac{d(\chi, \boldsymbol{\chi}_{i})}{h})\delta_{i}Z_{i}^{-\ell}-\mathbf{E}\left( K(\frac{d(\chi, \boldsymbol{\chi}_{i})}{h})\delta_{i}Z_{i}^{-\ell}\right)\right]}_{\mathcal{I}_{2.1}} + \underbrace{\frac{1}{\mathbb{E}\left(K(\frac{d(\chi, \boldsymbol{\chi}_)}{h})\right)}\mathbf{E}\left( K(\frac{d(\chi, \boldsymbol{\chi}_{i})}{h})\delta_{i}Z_{i}^{-\ell}\right)}_{\mathcal{I}_{2.2}}.&
\end{align*}
\end{small}
Proceeding as in the proof of Lemma~\ref{Lem2} one can show that $\mathcal{I}_{2.1}=\mathcal{O}\left(\sqrt{\frac{\log(n)}{n\phi_{\chi}(h)}}\right)$ a.s. Next,
\begin{small}
\begin{align*}
\mathcal{I}_{2.2}\leq \frac{\alpha^{-1}}{\mathbb{E}\left(K(\frac{d(\chi, \boldsymbol{\chi}_)}{h})\right)} \sup\limits_{a_{H} \leq t \leq b} L(t)\bar{G}(t)\iint K(\frac{d(\chi,u)}{h})y^{-\ell}f_{Y\mid \boldsymbol{\chi}}(y \mid u)dy d\mathbb{P}^{\boldsymbol{\chi}}(u)\leq\alpha^{-1} \sup\limits_{a_{H} \leq t \leq b} L(t)\bar{G}(t) C.
\end{align*}
\end{small}
Then \begin{small}$\mathcal{I}_{2.2}=\mathcal{O}\left(1\right)$\end{small},
so \begin{small}$\mathcal{I}_{2}=\mathcal{O}\left(1\right) \text{ a.s}$\end{small}. We conclude that\begin{small}$\mid\widehat{r}_{\ell,n}(\chi)-\tilde{r}_{\ell,n}(\chi)\mid=\mathcal{O}\left((n^{-1}\log\log(n))^{1/2})\right) \text{ a.s}.$\end{small} 
\end{proof}
\begin{flushleft}
Theorem~\ref{Theo1} is proved. $\square$
\end{flushleft}
\subsection*{Proof of Theorem 2}
To prove Theorem~\ref{Theo2} we need some additional lemmas.
\begin{lemma}
\label{LemU1}
Under assumptions ~\ref{Hyp:U1}-~\ref{Hyp:U3} $\ $ we have for $\ell=1,2$ :
\begin{small}
$$\sup_{\chi\in\mathcal{S}_{\mathcal{F}}}\mid \mathbf{E}(\tilde{r}_{\ell,n}(\chi))-r_{\ell}(\chi)\mid = \mathcal{O}\left(h^{k_{\ell}}\right).$$
\end{small}
\end{lemma}
\begin{proof}
Obtained immediately from Lemma~\ref{Lem1}.
\end{proof}
\begin{lemma}
\label{LemU3}
Under assumptions~\ref{Hyp:U1}-~\ref{Hyp:U3} we have for $\ell=1,2$:
\begin{small}
$$\sup_{\chi\in\mathcal{S}_{\mathcal{F}}}\mid\widehat{r}_{\ell,n}(\chi)-\tilde{r}_{\ell,n}(\chi)\mid=\mathcal{O}\left((n^{-1}\log\log(n))^{1/2})\right)\text{ a.s}.$$
\end{small}
\end{lemma}
\begin{proof}
By following the same steps as in the proof of Lemma~\ref{Lem3} and then taking the supremum over all $\chi$ in $\mathcal{S}_{\mathcal{F}}$ we get the desired result.
\end{proof}
\begin{lemma}
\label{LemU2}
Under assumptions ~\ref{Hyp:U1},~\ref{Hyp:U3} -~\ref{Hyp:U5} we have for $\ell=1,2$:
\begin{small}
$$\sup_{\chi\in\mathcal{S}_{\mathcal{F}}}\mid \tilde{r}_{\ell,n}(\chi)-\mathbf{E}(\tilde{r}_{\ell,n}(\chi)) \mid = \mathcal{O}\left(\sqrt{\frac{\psi_{\mathcal{S}_{\mathcal{F}}}\left(\frac{\log n}{n}\right)}{n\phi(h)}}\right) \text{ a.s}.$$ 
\end{small}
\end{lemma}
\begin{proof}
Let $\chi_{1}$, $\chi_{2}$, $\cdots$, $\chi_{N_{\epsilon}(\mathcal{S}_{\mathcal{F}})}$ a finite set of elements in $\mathcal{F}$ where $N_{\epsilon}(\mathcal{S}_{\mathcal{F}})$ represents the minimal number of open balls in $\mathcal{F}$ of radius $\epsilon$ needed to cover $\mathcal{S}_{\mathcal{F}}$, that is
$\mathcal{S}_{\mathcal{F}}\subset \bigcup_{k=1}^{N_{\epsilon}(\mathcal{S}_{\mathcal{F}})}B(\chi_{k},\epsilon)
\text{ and we take } \epsilon=\frac{\log n}{n}.
$
Set $k(\chi)= arg\min\limits_{k\in\left\lbrace 1,2,\cdots,N_{\epsilon}(\mathcal{S}_{\mathcal{F}})\right\rbrace} d(\chi, \boldsymbol{\chi}_{k})$ for all $\chi \in\mathcal{F}$ and $K_{i}\left(\chi\right):=K\left(\frac{d(\chi, \boldsymbol{\chi}_{i})}{h}\right)$.
We have
\begin{small}
\begin{align*}
\sup_{\chi \in\mathcal{S}_{\mathcal{F}}}\mid \tilde{r}_{\ell,n}(\chi)-\mathbf{E}(\tilde{r}_{\ell,n}(\chi)) \mid&\leq
\underbrace{\sup_{\chi\in\mathcal{S}_{\mathcal{F}}}\mid\tilde{r}_{\ell,n}(\chi)-\tilde{r}_{\ell,n}(\chi_{k(\chi)})}_{\mathcal{L}_{1}}\mid +\underbrace{\sup_{\chi\in\mathcal{S}_{\mathcal{F}}}\mid\tilde{r}_{\ell,n}(\chi_{k(\chi)})-\mathbf{E}\left(\tilde{r}_{\ell,n}(\chi_{k(\chi)})\right)}_{\mathcal{L}_{2}}\mid\\
&+\underbrace{\sup_{\chi\in\mathcal{S}_{\mathcal{F}}}\mid\mathbf{E}\left(\tilde{r}_{\ell,n}(\chi_{k(\chi)})\right)-\mathbf{E}(\tilde{r}_{\ell,n}(\chi))}_{\mathcal{L}_{3}} \mid.
\end{align*}
\end{small}
In order to achieve the proof we must evaluate the convergence rate of each term $\mathcal{L}_{j}$ for $j=1,2,3$. So
\begin{small}
$$
\mathcal{L}_{1}\leq \sup_{\chi\in\mathcal{S}_{\mathcal{F}}}\frac{1}{n}\sum_{i=1}^{n}\left\vert\dfrac{\alpha\delta_{i}K_{i}(\chi)Z_{i}^{-\ell}}{\mathbb{E}\left(K_{i}(\chi)\right)L(Z_{i})\bar{G}(Z_{i})}-\dfrac{\alpha\delta_{i}K_{i}(\chi_{k(\chi)})Z_{i}^{-\ell}}{\mathbb{E}\left(K_{i}(\chi_{k(\chi)})\right)L(Z_{i})\bar{G}(Z_{i})}\right\vert
$$
\end{small}
By combining assumptions~\ref{Hyp:U1} and~\ref{Hyp:U3} we get $\forall \chi\in \mathcal{S}_{\mathcal{F}}\text{, }0<C\phi\left(h\right)<\mathbb{E}\left(K_{i}(\chi)\right)<C^{\prime}\phi\left(h\right),$
then:
\begin{small}
\begin{align*}
\mathcal{L}_{1}&\leq \sup_{\chi\in\mathcal{S}_{\mathcal{F}}}\frac{c}{n\phi(h)}\sum_{i=1}^{n}\left\vert K_{i}(\chi)-K_{i}(\chi_{k(\chi)})\right\vert\times\frac{\alpha\delta_{i}Z_{i}^{-\ell}}{L(Z_{i})\bar{G}(Z_{i})}\leq \sup_{\chi\in\mathcal{S}_{\mathcal{F}}}\frac{c}{nh\phi(h)}\sum_{i=1}^{n}\left\vert d\left(\chi, \boldsymbol{\chi}_{i}\right)-d\left(\chi_{i},\chi_{k(\chi)}\right)\right\vert\times\frac{\alpha\delta_{i}Z_{i}^{-\ell}}{L(Z_{i})\bar{G}(Z_{i})}&\\
&\leq \sup_{\chi\in\mathcal{S}_{\mathcal{F}}}\frac{c}{n}\sum_{i=1}^{n}\dfrac{\epsilon\alpha\delta_{i}Z_{i}^{-\ell}}{h\phi(h)L(Z_{i})\bar{G}(Z_{i})}=:\sup_{\chi\in\mathcal{S}_{\mathcal{F}}}\frac{c}{n}\sum_{i=1}^{n}F_{i}.&
\end{align*}
\end{small}
Next, we evaluate $\mathbf{E}\left(F_{i}^m\right)$, $m\geq2$, for applying an exponential type inequality for unbounded random variables (see Corollary A8 in \cite{ferraty2006nonparametric}). We have under assumption~\ref{Hyp:U5}
\begin{small}
\begin{align*}
\mathbf{E}\left(F_{i}^{m}\right)&=\frac{\epsilon^{m}}{h^{m}\phi(h)^{m}}\mathbf{E}\left(\frac{\alpha^{m}\delta_{i}Z_{i}^{-m\ell}}{L(Z_{i})^{m}\bar{G}(Z_{i})^{m}}\right)=\frac{\epsilon^{m}}{h^{m}\phi(h)^{m}}\mathbb{E}\left(\frac{\alpha^{m-1}Y^{-m\ell}}{L(Y)^{m-1}\bar{G}(Y)^{m-1}}\right)\leq\frac{C\epsilon^{m}}{h^{m}\phi(h)^{m}}.&
\end{align*}
\end{small}
We apply the fore-mentioned exponential inequality with \begin{small}
$a^{2}=\dfrac{\epsilon}{h\phi(h)}=\frac{\log n}{nh\phi(h)}$\end{small} and Assumption~\ref{Hyp:U4} to get\begin{small}
$\mathcal{L}_{1}=\mathcal{O}\left((\frac{\psi_{\mathcal{S}_{\mathcal{F}}}\left(\frac{\log n}{n}\right)}{n\phi(h)})^{\frac{1}{2}}\right)$
\end{small}.\\
For the term $\mathcal{L}_{2}$, and for all $\eta>0$ we have:
\begin{small}
\begin{align*}
\mathbf{P}\left(\mathcal{L}_{2}>\eta(\frac{\psi_{\mathcal{S}_{\mathcal{F}}}(\epsilon)}{n\phi(h)})^\frac{1}{2}\right)\leq N_{\epsilon}(\mathcal{S}_\mathcal{F})\max\limits_{k(\chi)\in \left\lbrace1,2,\cdots,N_{\epsilon}\right\rbrace}\mathbf{P}\left(\mid\tilde{r}_{\ell,n}(\chi_{k(\chi)})-\mathbf{E}(\tilde{r}_{\ell,n}(\chi_{k(\chi)}))\mid>\eta\sqrt{\frac{\psi_{\mathcal{S}_{\mathcal{F}}}\left(\epsilon\right)}{n\phi(h)}}\right).&
\end{align*}
\end{small}
By using the same arguments as in Lemma~\ref{Lem2} (taking $a^{2}=\phi(h)^{-m+1}$) we can show that:
\begin{small}
$$
\mathbf{P}\left(\left\vert\tilde{r}_{\ell,n}(\chi_{k(\chi)})-\mathbf{E}\left(\tilde{r}_{\ell,n}(\chi_{k(\chi)})\right)\right\vert>\eta\sqrt{\frac{\psi_{\mathcal{S}_{\mathcal{F}}}\left(\epsilon\right)}{n\phi(h)}}\right)\leq 2\exp\left(-C\eta^{2}\psi_{\mathcal{S}_\mathcal{F}}\left(\epsilon\right)\right).
$$
\end{small}
We have $\psi_{\mathcal{S}_\mathcal{F}}\left(\epsilon\right)=\log N_{\epsilon}(\mathcal{S}_\mathcal{F})$  as defined in \cite{ferraty2010rate} then:
\begin{small}
\begin{align*}
\mathbf{P}\left(\left\vert\tilde{r}_{\ell,n}(\chi_{k(\chi)})-\mathbf{E}\left(\tilde{r}_{\ell,n}(\chi_{k(\chi)})\right)\right\vert>\eta\sqrt{\frac{\psi_{\mathcal{S}_{\mathcal{F}}}\left(\epsilon\right)}{n\phi(h)}}\right)&\leq 2\exp\left(-C\eta^{2}\log  N_{\epsilon}(\mathcal{S}_\mathcal{F})\right)\leq 2N_{\epsilon}(\mathcal{S}_\mathcal{F})^{-C\eta^{2}},&
\end{align*}
\end{small}
it follows that:
\begin{small}
\begin{align*}
N_{\epsilon}(\mathcal{S}_\mathcal{F})\max\limits_{k(\chi)\in \lbrace1,2,\cdots,N_{\epsilon}\rbrace}&\mathbf{P}\left(\mid\tilde{r}_{\ell,n}(\chi_{k(\chi)})-\mathbf{E}\left(\tilde{r}_{\ell,n}(\chi_{k(\chi)})\right)\mid>\eta\sqrt{\frac{\psi_{\mathcal{S}_{\mathcal{F}}}\left(\epsilon\right)}{n\phi(h)}}\right)\leq 2N_{\epsilon}(\mathcal{S}_\mathcal{F})^{1-C\eta^{2}}.&
\end{align*}
\end{small}
It is sufficient to take $\eta=\sqrt{\frac{\gamma}{C}}$ which under assumption~\ref{Hyp:U4} and by Borel-Cantelli's lemma gives:
\begin{small}
\begin{align*}
\sum_{n=1}^{\infty}\mathbf{P}\left(\mathcal{L}_{2}>\eta\sqrt{\frac{\psi_{\mathcal{S}_{\mathcal{F}}}\left(\epsilon\right)}{n\phi(h)}}\right)&\leq2\sum_{n=1}^{\infty} N_{\epsilon}(\mathcal{S}_\mathcal{F})^{1-\gamma}<\infty \Rightarrow \mathcal{L}_{2}=\mathcal{O}\left(\sqrt{\frac{\psi_{\mathcal{S}_{\mathcal{F}}}\left(\epsilon\right)}{n\phi(h)}}\right)
\end{align*}
\end{small}
Regarding the term $\mathcal{L}_{3}$ it is obvious by Jensen's inequality that:
\begin{small}
$$\mathcal{L}_{3}\leq\mathbf{E}\left(\sup_{\chi\in\mathcal{S}_{\mathcal{F}}}\mid\tilde{r}_{\ell,n}(\chi)-\tilde{r}_{\ell,n}(\chi_{k(\chi)})\mid\right),$$
\end{small}
then the proof for the term $\mathcal{L}_{1}$ remains valid for showing that \begin{small}
$\mathcal{L}_{3}=\mathcal{O}\left(\sqrt{\dfrac{\psi_{\mathcal{S}_{\mathcal{F}}}\left(\epsilon\right)}{n\phi(h)}}\right)$
\end{small}  since $\epsilon=\frac{\log n}{n}$, this concludes the proof of Lemma~\ref{LemU2}.
\end{proof}
\begin{flushleft}
The proof of Theorem~\ref{Theo2} is now complete. $\square$
\end{flushleft}

\bibliographystyle{unsrtnat}
\bibliography{references}

\end{document}